\documentclass[12pt,oneside,english]{amsart}
\usepackage[T1]{fontenc}
\usepackage{amsmath}
\usepackage{amsthm}
\usepackage{amsfonts}
\usepackage{amssymb}
\usepackage{graphicx}
\usepackage{mathrsfs}
\usepackage{amscd}
\usepackage[pagebackref=false,colorlinks,linkcolor=blue,citecolor=blue]{hyperref}

\DeclareMathOperator{\Prob}{\mathbf{P}}   
\DeclareMathOperator{\E}{\mathbf{E}}     
\DeclareMathOperator{\Var}{\mathbf{Var}}      
\DeclareMathOperator{\Disc}{\mathbf{D}}

\newcommand\eps{{\varepsilon}}

\newcommand{\CC}{{\mathbb C}}
\newcommand{\RR}{{\mathbb R}}

\newcommand{\ZZ}{{\mathbb{Z}}}
\newcommand{\Sphere}{{\mathbb{S}^2}}
\newcommand{\modifiedX}{{\tilde{\mathcal{X}}}}

\renewcommand\Re{{\operatorname{Re}}}

\newcommand{\dist}{\overset{\mbox{\rm \scriptsize d}}{\sim}}
\newcommand{\todist}{\overset{\mbox{\rm \scriptsize d}}{\longrightarrow}}

\newcommand{\dd}{{\mathrm{d}}}            
\newcommand{\Id}{{\mathrm{Id}}}
\renewcommand{\L}{{\mathrm{L}}}

\newtheorem{thm}{Theorem}[section]
\newtheorem{cor}[thm]{Corollary}

\newtheorem{lem}[thm]{Lemma}
\newtheorem*{lem*}{Lemma}
\newtheorem{prop}[thm]{Proposition}
\newtheorem*{prop*}{Proposition}
\newtheorem*{theo*}{Theorem}
\theoremstyle{definition}

\theoremstyle{remark}
\newtheorem*{rem}{Remark}

\numberwithin{equation}{section}

\begin{document}
\title[The spherical ensemble] 
{The spherical ensemble and uniform distribution of points on the sphere}

\author{Kasra Alishahi}
 \address{Department of mathematics, Sharif University of Technology, Tehran, Iran}
 \email{alishahi@sharif.edu}

\author{Mohammad Sadegh Zamani}
 \address{Department of mathematics, Sharif University of Technology, Tehran, Iran}
 \email{ms\_zamani@mehr.sharif.edu}

\begin{abstract}
The spherical ensemble is a well-studied determinantal process with a fixed number of points on $\Sphere$. The points of this process correspond to the generalized eigenvalues of two appropriately chosen random matrices, mapped to the surface of the sphere by stereographic projection. This model can be considered as a spherical analogue for other random matrix models on the unit circle and complex plane such as the circular unitary ensemble or the Ginibre ensemble, and is one of the most natural constructions of a (statistically) rotation invariant point process with repelling property on the sphere.

In this paper we study the spherical ensemble and its local repelling property by investigating the minimum spacing between the points and the area of the largest empty cap. Moreover, we consider this process as a way of distributing points uniformly on the sphere. To this aim, we study two "metrics" to measure the uniformity of an arrangement of points on the sphere. For each of these metrics (discrepancy and Riesz energies) we obtain some bounds and investigate the asymptotic behavior when the number of points tends to infinity. It is remarkable that though the model is random, because of the repelling property of the points, the behavior can be proved to be as good as the best known constructions (for discrepancy) or even better than the best known constructions (for Riesz energies).
\end{abstract}

\maketitle
\setcounter{footnote}{1}
\section{{\bf Introduction}}
\subsection{Background}
The aim of this paper is to study the statistical properties of a natural point process on the sphere where the points exhibit repulsive behavior. This point process was introduced in \cite{Krishnapur: thesis} and is known as spherical ensemble; see \cite{GAF-Book} and  \cite{Krishnapur}. The model was studied earlier in \cite{Caillol} and \cite{Forrester}, but
without observing the connection to random matrices. It was shown in \cite{Caillol,Forrester} that there exists a connection between this model and the one-component plasma on the sphere for a special value of temperature. See \cite{Book: Forrester} for further discussion.

Let $A_n$ and $B_n$ be independent $n\times n$ random matrices with independent and identically distributed standard complex Gaussian entries, and let $\{\lambda_1,\lambda_2, \ldots, \lambda_n\}$ denotes the set of eigenvalues of $A_n^{-1}B_n$. We can consider these eigenvalues as a (simple) random point process on complex plane $\CC$. The point process $\{\lambda_1,\ldots,\lambda_n\}$ can be described using the {\it $k$-point correlation functions $\rho_k^{(n)}: \CC^k \to \RR^{\geq 0}$}, $1 \leq k \leq n$, defined in such a way that
\begin{align} \label{eqn: def of rho}
\int_{\CC^k} &F(z_1,\ldots,z_k) \rho_k^{(n)}(z_1,\ldots,z_k) \,\dd \mu(z_1)\ldots \dd \mu(z_k)\\
&=\E \sum_{\substack{i_1,\ldots,i_k \\ \mbox{\scriptsize{\textit{pairwise distinct}}}}} F(\lambda_{i_1},\ldots,\lambda_{i_k}), \nonumber
\end{align}
for all continuous, compactly supported functions $F:\CC^k \to \CC$, where $\dd \mu(z):=\frac{n}{\pi (1+|z|^2)^{n+1}} \dd z$ and $\dd z$ denotes the Lebesgue measure on the complex plane $\CC$.

Krishnapur \cite{Krishnapur} showed that this random point process is a determinantal point process on complex plane with kernel
\[K^{(n)}(z,w)=(1+z\bar{w})^{n-1}\]
with respect to the background measure $\dd \mu(z)$, i.e. we have
\begin{equation} \label{eqn: Determinantal form}
\rho_k^{(n)}(z_1,\ldots,z_k)=\det \left[K^{(n)}(z_i,z_j)\right]_{i,j=1}^{k}
\end{equation}
for every $k \geq 1$ and $z_1,\ldots,z_k \in \CC$. 
We note here that a random point process is said to be a {\it determinantal point process} if its $k$-point correlation functions have determinantal form similar to (\ref{eqn: Determinantal form}). The corresponding kernel $K^{(n)}(z,w)$ is called a correlation kernel of the determinantal point process. We refer to \cite{GAF-Survey} or \cite{GAF-Book} and references therein  for more information on deteminantal point processes.   

Let $\Sphere=\{p \in \RR^3: |p|=1\}$ be the unit two-dimensional sphere centred at the origin in three-dimensional Euclidean space $\RR^3$. Also we let $\nu$ denotes the Lebesgue surface area measure on this sphere with total measure $4 \pi$. As mentioned in \cite{GAF-Book}, these eigenvalues are best thought of as points on $\Sphere$, using stereographic projection.  Let $g$ be the stereographic projection of the sphere $\Sphere$ from the north pole onto the plane $\{(t_1,t_2,0); t_1,t_2 \in \RR\}$. If we let $P_i=g^{-1}(\lambda_i)$ for $1 \leq i \leq n$ then the vector $(P_1,\ldots,P_n)$, in uniform random order, has the joint density
\[\mbox{Const.} \prod_{i<j} |p_i-p_j|^2\] 
with respect to Lebesgue measure on $(\Sphere)^n$ where $|p_i-p_j|$ denotes the Euclidean distance between two points $p_i$ and $p_j$.     Note that this density is similar to the circular unitary ensemble case and clearly this point process is invariant in distribution under the isometries of $\Sphere$.
Consider the point process on $\Sphere$,
 \[\mathcal{X}^{(n)}:=\sum_{j=1}^n \delta_{P_j}.\]
 We know that $\frac{1}{n} \mathcal{X}^{(n)}$ converges almost surely to the uniform measure on the sphere. In fact, it is also true in the more general case:
Let $A'_n$ and $B'_n$ be independent $n\times n$ random matrices with i.i.d entries with mean 0 and variance 1, and let $\{\lambda'_1,\lambda'_2, \ldots, \lambda'_n\}$ denotes the set of eigenvalues of $A_{n}^{'-1} B'_n$.
Based on the results of \cite{Tao}, Bordenave \cite{Bordenave} shows that 
$\frac{1}{n} \sum_{j=1}^n \delta_{g^{-1}(\lambda'_j)}$ 
converges almost surely to uniform measure on $\Sphere$ as $n \to +\infty.$

Moreover from the repulsive nature of determinantal point processes we expect that the points of process $\mathcal{X}^{(n)}$ are typically more evenly distributed than $n$ independently chosen uniform points on the sphere.
This repelling property is the common feature of many models in random matrix theory and has been comprehensively studied in some special models such as the Gaussian unitary or the circular unitary ensembles. For example, the distribution or the minimum or the maximum of the gaps between consecutive eigenvalues have been computed and compared to simpler models as a way to measure and understand the repulsive structure. One of the goals of this paper is to do the same computations for the spherical model. We specially focus on the minimum distance between the points, the area of the largest empty cap, the hole probability and the limiting distribution of the nearest neighbors distances as natural two-dimensional extensions of the so called metrics studied in one dimensional models.

On the other hand we exploit this model and its properties to the classic problem of distributing points on the sphere. The problem of distributing a given number of points on the surface of a sphere "uniformly", is a challenging and old problem. Contrary to the one dimensional case (i.e. distributing points on a circle) where the most uniform arrangement clearly exists and is attained when the points are equidistributed, it seems that there is no arrangement on the sphere that can be considered as completely uniform, and the answer for the best arrangement depends on what criteria do we use to quantify the uniformity of an arrangement. Among the mostly used criteria are those related to the electrostatic potential energy and its generalizations where one tries to distribute the points in a way that minimizes some energy function. Another common metric is the discrepancy that measures the maximum deviance between the number of points and the expected number, in some class of regions in the underlying space (sphere, in our case). Both the energy and the discrepancy optimization problems, i.e. finding the most optimum arrangement or even obtaining some relatively sharp upper and lower bounds are open and challenging problems for the sphere. We study these metrics (discrepancy and Riesz energies) to measure the uniformity of points of $\mathcal{X}^{(n)}$. For each of these metrics  we obtain some bounds and investigate the asymptotic behavior when the number of points tends to infinity. It is remarkable that though the model is random, because of the repelling property of the points, the behavior can be proved to be as good as the best known constructions (for discrepancy) or even better than the best known constructions (for Riesz energies). 
 
 The main results of the paper are stated in the next subsection, together with definition and some properties of the metrics discussed above. 
\subsection{Main results}
\subsubsection*{{\bf Discrepancy}}
The geometrically most natural measure for the uniformity of the distribution of an $n$-point set on $\Sphere$ is the {\it spherical cap discrepancy}. 
Let $\mathcal{P}_n=\{x_1,\ldots,x_n\}$ be an $n$-point set on $\Sphere$. The spherical cap discrepancy of $\mathcal{P}_n$ is defined as
\[\mathbf{D}(\mathcal{P}_n):=\sup_{D \in \mathcal{A}} \left| \sum_{j=1}^{n} 1_D(x_j)-\frac{n |D|}{4\pi} \right| \]
where  $\mathcal{A}$ is the set of all spherical caps on $\Sphere$. A spherical cap is defined as the intersection of the sphere and a half-space.   
In \cite{Beck}, it was shown that there is constant $c>0$, independent of $n$, such that for any $n$-point set $\mathcal{P}_n$ on $\Sphere$  we have
\[cn^{1/4}\leq \mathbf{D}(\mathcal{P}_n).\]
On the other hand, using probabilistic methods  it has been shown (see \cite{Beck2}) that for any $n \geq 1$, there exists $n$-point set $\mathcal{P}_n$ on $\Sphere$ such that 
\[\mathbf{D}(\mathcal{P}_n) \leq Cn^{1/4} \sqrt{\log n}.\]
 
The following theorem shows that the point process $\mathcal{X}^{(n)}$  has small spherical cap discrepancy.

\begin{thm} \label{thm: Discrepancy}
Consider the point process $\mathcal{X}^{(n)}=\sum_{i=1}^n \delta_{P_i}$. For every $M>0$ independent of $n$, one has 
\[\mathbf{D}(\{P_1,\ldots,P_n\})=O\left(n^{1/4}\sqrt{\log n} \right)\]     
with probability $1-\frac{1}{n^M}.$
\end{thm}
Note that for independent uniform points on sphere, the discrepancy is of order $\sqrt{n}$ (up to a logarithmic factor).  

The key to the proof of Theorem \ref{thm: Discrepancy} is an estimate on the variance of the number of points of $\mathcal{X}^{(n)}$ on a spherical cap. The asymptotic expansion of this variance and the proof of Theorem \ref{thm: Discrepancy} will be presented in Section 2.
\subsubsection*{{\bf Largest empty cap}}
Given $\mathcal{P}_n=\{x_1,\ldots,x_n\} \subset \Sphere$, define the covering radius of $\mathcal{P}_n$ as the infimum of the numbers $t>0$ such that every point of $\Sphere$  is within distance $t$ of some $x_j$. If we let $\tau$ be the covering radius of $\mathcal{P}_n$, then the area of the largest spherical cap which does not contain any point of $\mathcal{P}_n$ in its interior is equal to $\pi \tau^2$ (note that, for fixed $q$, the area of the spherical cap $\{p \in \Sphere: |p-q| \leq r\}$ is $\pi r^2$).  
We will be interested in studying the asymptotic behavior of the area of the largest empty cap for the spherical ensemble. 

Let $M_n$ be the area of largest empty cap for random point process $\mathcal{X}^{(n)}$. In the following theorem, we derive first-order asymptotic for $M_n$.  
\begin{thm} \label{thm: Largest empty cap}
For any $s>0$ we have
\[\frac{n}{8\pi \sqrt{\log n}}M_n \overset{L^s}{\longrightarrow} 1\]
as $n \to +\infty$.
\end{thm}
The proof of this theorem is given in Section 3. For the proof we need asymptotics of the {\it hole probability}, the probability that there are no points of $\mathcal{X}^{(n)}$ in a given spherical cap. The desired asymptotics of the hole probability will be given in Lemma \ref{Prop: Hole probability}. Then we will prove Theorem \ref{thm: Largest empty cap} using a method similar to the proof of Theorem 1.3 in \cite{Arous}. Notice that for independent uniform points on $\Sphere$, the area of largest empty cap is of order $\frac{\log n}{n}$.

At the end of Section 3 we study the nearest neighbour statistics of spherical ensemble and show a connection between the local behavior of this model and the Ginibre ensemble.
\subsubsection*{\bf{Riesz and logarithmic energy}}
In Section 4, we compute the expectations of the logarithmic energy and Riesz $s$-energy of the random point process $\mathcal{X}^{(n)}$ on $\Sphere$. 
The discrete logarithmic energy of $n$ points $x_1,\ldots,x_n$ on $\Sphere$ is given by
\begin{equation*}
E_{\log}(x_1,\ldots,x_n):=\log \prod_{i\neq j} \frac{1}{|x_i-x_j|}=\sum_{i \neq j} \log \frac{1}{|x_i-x_j|}.
\end{equation*}
Also we define 
\begin{equation*}
\mathcal{E}_{\log}(n):=\min \{E_{\log}(x_1,\ldots ,x_n); x_1,\ldots ,x_n \in \Sphere\}.
\end{equation*}
The n-tuples that minimize this energy are usually called elliptic Fekete Points. 
Define $C_n$ by 
\[\mathcal{E}_{\log}(n)=\left(\frac{1}{2}-\log 2\right)n^2-\frac{1}{2} n\log n+C_n n.\]
In \cite{RSZ94} it was shown that $C_n$ satisfies the following estimates
\begin{equation*}
-0.22553754 \dots \leq \liminf_{n \to +\infty} C_n \leq \limsup_{n \to +\infty} C_n \leq -0.0469945 \dots.
\end{equation*}

For a given $s$, the \emph{Riesz $s$-energy} of $n$ points $x_1,\ldots,x_n$ on $\Sphere$ are defined as
\begin{equation*}
E_{s}(x_1,\ldots,x_n):=\sum_{i \neq j}  \frac{1}{|x_i-x_j|^s}.
\end{equation*}
Also, we consider the optimal $n$-point Riesz $s$-energy,
\begin{equation*}
\mathcal{E}_{s}(n):= \left\{
\begin{array}{rl}
\min \{E_{s}(x_1,\ldots,x_n); x_1,\ldots,x_n \in \Sphere\} & \text{if } s \geq 0\\ \\
\max \{E_{s}(x_1,\ldots,x_n); x_1,\ldots,x_n \in \Sphere\} & \text{if } s<0. 
\end{array} \right.
\end{equation*}
The important special case $s=1$ corresponds to  electrostatic potential energy of electrons on $\Sphere$  that repel each other with a force given by Coulomb's law. We remark that this problem is only interesting for $s > -2$.
It is known that for the {\it potential-theoretic regime}, $-2<s<2$, we have
\begin{equation} \label{eqn: Energy first order}
\lim_{n \to +\infty} \frac{\mathcal{E}_{s}(n)}{n^2}=\frac{1}{(4 \pi)^2}\int_{\Sphere\times \Sphere}  \frac{1}{|p-q|^s} \,\dd \nu(p)\,\dd \nu(q)=\frac{2^{1-s}}{2-s}.
\end{equation}
 See e.g. \cite{BHS2012}. Consider the difference $\mathcal{E}_s(n)-\frac{2^{1-s}}{2-s}n^2$. Wagner (\cite{WagnerLower} lower bound for $0<s<2$ and upper bound for $-2<s<0$, and \cite{WagnerUpper} upper bound for $0<s<2$ and lower bound for $-2<s<0$) proved that 
\[-c_1 n^{1+s/2} \leq \mathcal{E}_s(n)-\frac{2^{1-s}}{2-s}n^2 \leq -c_2 n^{1+s/2} \quad, \quad -2<s<2\]
where $c_1$ and $c_2$ are positive constants depending only on $s$.  In \cite{RSZ94}, an alternative method which is based on partitioning
$\Sphere$ into regions of equal area and small diameter is used to prove the upper bound in the case $0<s<2$ (and lower bound in the case $-2<s<0$). Let $\eps>0$ be arbitrary, this method gives
\begin{align}
\mathcal{E}_{s}(n) - \frac{2^{1-s}}{2-s}n^2 &\leq -(2 \sqrt{2 \pi})^{-s}(1-\eps)n^{1+s/2} \quad , \quad 0<s<2 \label{Energy: upper bound}   \\  
\mathcal{E}_{s}(n) - \frac{2^{1-s}}{2-s}n^2 &\geq -(2 \sqrt{2 \pi})^{-s}(1+\eps)n^{1+s/2} \quad , \quad -2<s<0 \label{Energy: lower bound} 
\end{align}
for $n \geq n_0(\eps,s)$. We show that the better bounds than (\ref{Energy: upper bound}) and (\ref{Energy: lower bound}) can be obtained by considering the expectation of Riesz $s$-energy of point process $\mathcal{X}^{(n)}$. It is conjectured (see \cite{BHS2012}, Conjecture 3) that the asymptotic expansion of the optimal Riesz $s$-energy for $-2<s<4$, $s\neq 2$ has the form
\begin{equation} \label{conjecture: Riesz Energy}
\mathcal{E}_{s}(n)=\frac{2^{1-s}}{2-s}n^2+\frac{(\sqrt{3}/2)^{s/2} \zeta_{\Lambda_2}(s)}{(4\pi)^{s/2}} n^{1+s/2}+o(n^{1+s/2}) \quad n \to +\infty,
\end{equation}
where $\zeta_{\Lambda_2}(s)$ is the zeta function of the hexagonal lattice $\Lambda_2=\{m(1,0)+n(1/2,\sqrt{3}/2): m, n \in \ZZ\}$.
 See the survey \cite{BHS2012} for more details and further discussion.

In the boundary case $s=2$,  we have (see Theorem 3 in \cite{KS98})
\[\lim_{n \to +\infty} \frac{\mathcal{E}_{2}(n)}{n^2 \log n}=\frac{1}{4}.\]
Also, in \cite{BHS2012} (see Proposition 3 and its remark therein), it is shown that
\begin{equation} \label{Energy: upper bound for s=2}
-\frac{1}{4}n^2+O(n) \leq \mathcal{E}_{2}(n)-\frac{1}{4}n^2 \log n \leq \frac{1}{4} n^2 \log \log n+O(n^2).
\end{equation}
Considering the expectation of Riesz $s$-energy of point process $\mathcal{X}^{(n)}$, we are able to omit the $\log \log n$ term in this  estimate. (See Conjecture 5 in \cite{BHS2012} for  the asymptotic expansion of the optimal Riesz 2-energy.) 
 We remark that the first term of the asymptotics of $\mathcal{E}_s(n)$ for $s>2$ is not known. 

In the following theorem, we give the expectations of the logarithmic energy and Riesz $s$-energy of the random point process $\mathcal{X}^{(n)}$ on $\Sphere$.
\begin{thm} \label{theorem: Riesz energy}
For the point process $\mathcal{X}^{(n)}$ on $\Sphere$, $n \geq 2$, we have

i) (Logarithmic energy)
\begin{equation} \label{thm: Energy part 1}
\E E_{\log}(P_1,\ldots,P_n)=\left(\frac{1}{2}-\log 2\right)n^2-\frac{1}{2} n\log n+\left(\log 2- \frac{\gamma}{2}\right) n-\frac{1}{4}+O\left(\frac{1}{n}\right)
\end{equation}
Here, $\gamma$ is the Euler constant. \\

ii) (Riesz $s$-energy: $s<4$ and $s \neq 2$)
\begin{equation} \label{thm: Energy part 2}
\E E_{s}(P_1,\ldots , P_n)=\frac{2^{1-s}}{2-s}n^2-\frac{\Gamma(n)\Gamma(1-s/2)}{2^{s}\Gamma(n+1-s/2)}n^2
\end{equation}\\[-5mm]

iii) (Riesz $s$-energy: $s=2$) 
\begin{equation} \label{thm: Energy part 3}
\E E_{2}(P_1,\ldots , P_n)=\frac{1}{4}n^2 \log n+\frac{\gamma}{4}n^2-\frac{n}{8}-\frac{1}{48}+O\left(\frac{1}{n^2}\right)
\end{equation}
\end{thm}
As a corollary to Theorem \ref{theorem: Riesz energy}, we obtain the following bounds for optimal Riesz $s$-energy. 

\begin{cor} \label{Cor: Riesz energy}
for every $n \geq 2$,
\begin{align} \label{Energy: new upper bound}
\mathcal{E}_{s}(n) - \frac{2^{1-s}}{2-s}n^2 &\leq -\frac{\Gamma(1-s/2)}{2^s} n^{1+s/2} \quad , \quad \,\, 0<s<2 \\ \label{Energy: new lower bound} 
\mathcal{E}_{s}(n) - \frac{2^{1-s}}{2-s}n^2 &\geq -\frac{\Gamma(1-s/2)}{2^s} n^{1+s/2} \quad , \quad -2<s<0 \\ \label{Energy: new upper bound for s=2} 
 \mathcal{E}_{2}(n)-\frac{1}{4}n^2 \log n &\leq \frac{\gamma}{4} n^2 
\end{align}
\end{cor}

Suppose that $x_1,\ldots,x_n$ are chosen randomly and independently on the sphere, with the uniform distribution. One can easily show that 
\[\E E_{\log}(x_1,\ldots,x_n)=\left(\frac{1}{2}-\log 2\right)n^2- \left(\frac{1}{2}-\log 2\right) n\]
and also for $s<2$ (see (\ref{eqn: Energy first order}))
\[\E E_{s}(x_1,\ldots,x_n)=\frac{2^{1-s}}{2-s}n^2-\frac{2^{1-s}}{2-s} n.\]

\subsubsection*{{\bf Minimum spacing}}
Define the minimum spacing by
\[m_n:=\min_{i \neq j} |P_i-P_j|.\]
We are interested in the asymptotic behavior of $m_n$ as $n$ tends to infinity.
For $0<t<2$, set
\begin{equation} \label{def: point pair statistics}
G_{t,n}:=\sum_{i<j} 1_{\{|P_i-P_j| \leq t\}}
\end{equation}
to be the number of non-ordered pairs of distinct points at Euclidean distance at most $t$ apart. Our result concerning distribution of $G_{t,n}$ is the following.
\begin{thm} \label{thm: Minimum spacing}
Let $G_{t,n}$ defined by (\ref{def: point pair statistics}) and assume that $t=\frac{x}{n^{3/4}}$ then $G_{t,n}$ converges in distribution to the Poisson random variable
with mean $\frac{x^4}{64}$.
\end{thm}
The proof of this theorem is given in Section 5. As a consequence, since $\Prob(G_{t,n}=0)=\Prob(m_n>t)$, Theorem \ref{thm: Minimum spacing} clearly implies that
\begin{cor}
For any $x>0$,
\[\lim_{n \to +\infty} \Prob(n^{3/4}m_n >x)=\exp\left(\frac{-x^4}{64}\right).\]
\end{cor}
Suppose that $x_1,\ldots,x_n$ are chosen randomly and independently on the sphere, with the uniform distribution.
It was shown that (see Theorem 2 of \cite{Randomsphere})
\[\lim_{n \to +\infty} \Prob(n \min_{i \neq j} |x_i-x_j|>x)=\exp\left(\frac{-x^2}{8}\right).\]

\section{{\bf Discrepancy}}
Let $D$ be a spherical cap on the sphere $\Sphere$. Define 
\[N_{D}:=\mathcal{X}^{(n)}(D),\]
the number of points of $\mathcal{X}^{(n)}$ in $D$. Clearly, the expected value of $N_D$ is equal to $\frac{n|D|}{4\pi}$. In order to prove Theorem \ref{thm: Discrepancy} we need to control the variance of $N_D$. The following lemma gives the asymptotic behavior of the variance.
\begin{lem}[] \label{lem:variance}
 Let $D$ be a spherical cap on the sphere (depending on $n$) such that $\frac{1}{|D|},\frac{1}{|D^c|}=o(n)$. Then for any $\eps>0$
\begin{equation} \label{appro: variance}
\Var [N_D]=\frac{\sqrt{n}}{4\pi \sqrt{\pi}}\sqrt{|D||D^c|}+ o\left(\log^{\frac{1}{2}+\eps} (n|D||D^c|)\right)
\end{equation} 
where $|D|$ denotes the area of $D$ and $D^c=\Sphere-D$.
\end{lem}   

\begin{proof}

The distribution of $N_D$ is invariant under isometries of the sphere, so we may assume without loss of generality that $g(D)$ is a disk centred
 at the origin with radius $r$. From \cite{GAF-Survey}, Theorem 26, the set $\{|\lambda_k|^2: 1 \leq k \leq n\}$ has the same distribution as 
 set $\{Q_k:0\leq k\leq n-1\}$ where the random variables $Q_k$ are jointly independent and $Q_k$ for $0\leq k\leq n-1$ has density
\[\frac{n\binom{n-1}{k}q^k}{(1+q)^{n+1}} \quad , \quad q\geq 0\]
(Note that $K(z,w)=\sum_{k=0}^{n-1} \binom{n-1}{k}(z\bar{w})^k$ and $\mu$ has density $\phi(|z|)$ where $\phi(x)=\frac{n}{\pi (1+x^2)^{n-1}}$).  
Thus $Q_k$ has the beta prime distribution with parameters $k+1$ and $n-k$, i.e.,  $\frac{Q_k}{1+Q_k}$  has the beta distribution with parameters 
as before. 

Thus we have
\begin{equation} \label{dist0: N_D}
N_D \dist |\{k: 0 \leq k \leq n-1 \, , \, Q_k<r^2\}|.
\end{equation}

And so by the independence of $Q_k$, $0 \leq k \leq n-1$, we deduce that
\begin{equation} \label{dist1: N_D}
N_D \dist \eta_0+\eta_1+\dots+\eta_{n-1}
\end{equation}   
where the random variables $\eta_k$ are independent and distributed as Bernoulli variables with $\Prob(\eta_k=1)=\Prob(Q_k<r^2)$.  
By the properties of incomplete beta function (see e.g. \cite{beta}) we have 
\begin{equation*}
\Prob(Q_k<r^2)=\int_{0}^{r^2} \frac{n\binom{n-1}{k}q^k}{(1+q)^{n+1}}\,\mathrm{d}q
=\frac{\sum_{j=k+1}^{n}\binom{n}{j} r^{2j}}{(1+r^2)^n}
\end{equation*}
 Now suppose $B_1,\ldots,B_n$ are i.i.d. Bernoulli random variables with parameter $\alpha:=\frac{r^2}{1+r^2}$.
 Let $S_n:=B_1+\dots+B_n$. One can write the right-hand side of above equation as $\Prob(S_n>k)$. Therefore, we obtain for every $0 \leq k \leq n-1$
\begin{equation} \label{eqn: Bernoulli}
\Prob(\eta_k=1)=\Prob(S_n>k).
\end{equation}
We also note that
\begin{equation*}
|D|=\frac{4\pi r^2}{1+r^2}=4\pi \alpha .
\end{equation*}
Now from (\ref{dist1: N_D}) and (\ref{eqn: Bernoulli}) it follows that for every spherical cap $D$,
\begin{equation} \label{eqn: variance}
\Var[ N_D]=\sum_{k=0}^{n-1} \Prob(S_n \leq k) \Prob(S_n>k)
\end{equation}
where $S_n$ is a binomial random variable with parameters $n$ and $\frac{|D|}{4\pi}$.

Fix $\eps>0$. Let 
\[\sigma=\sqrt{\Var[S_n]}=\sqrt{n\alpha(1-\alpha)}\] and 
\[M=\sigma (\log \sigma)^{\frac{1}{2}+\eps}.\]
Since $\Var[N_D]=\Var[N_{D^c}]$, by symmetry, we may assume that $\alpha \leq 1/2$. 
Using the Bernstein's inequality we see that for any $t>0$, one has
\begin{equation} \label{inequality: Bernstein}
\Prob (|S_n-n\alpha|>t ) \leq 2\exp\left(-\min\left(\frac{t^2}{4\sigma^2},\frac{t}{4}\right)\right)
\end{equation}
(see for example Lemma 2 in \cite{Beck2}).
 We can use this to show that
\begin{align} \label{estimate: Chernoff}
&\sum_{|k-n\alpha |\geq M} \Prob(S_n \leq k) \Prob(S_n>k) \leq \\
 &2 \left[ \sum_{M \leq j \leq \sigma^2} \exp\left(-\frac{j^2}{4\sigma^2}\right)+\sum_{\sigma^2 <j\leq n(1-\alpha)} \exp(-j/4) \right]. \nonumber
\end{align}
The assumption of the lemma implies that $\sigma \to +\infty$ as $n \to +\infty$ and thus we have
\begin{equation} \label{appro2: variance}
\sum_{\sigma^2 <j\leq n(1-\alpha)} \exp(-j/4)=o(1).
\end{equation}
Also, by comparing the integral of $\exp\left(\frac{-x^2}{4}\right)$ with its Riemann sum, we conclude that
 \begin{align*}
 &\sum_{M+1 \leq j \leq \sigma^2} \exp\left(-\frac{j^2}{4\sigma^2}\right) \leq \sigma \int_{(\log \sigma)^{1/2+\eps}}^{+\infty} \exp\left(\frac{-x^2}{4}\right)\,\dd x \\ 
 &= \sigma \int_{(\log \sigma)^{1+2\eps}}^{+\infty} \frac{1}{2\sqrt{x}} \exp(-x/4)\,\dd x \\
 &\leq \frac{2\sigma}{(\log \sigma)^{1/2+\eps}} \exp\left(-\frac{1}{4}(\log \sigma)^{1+2\eps}\right)=o(1).
 \end{align*}
Thus, from (\ref{estimate: Chernoff}), (\ref{appro2: variance}) and above estimate, we obtain 
\[\sum_{|k-n\alpha|\geq M} \Prob(S_n \leq k) \Prob(S_n>k)=o(1).\]
Moreover, from the Berry-Ess\'een theorem we see that for some absolute constant $C$, one has (see \cite{Berry}) 
\[\left|\Prob(S_n \leq k)-\Phi\left(\frac{k-n\alpha}{\sigma}\right)\right|\leq \frac{C}{\sigma}\]
where $\Phi$ is the cumulative distribution function of the standard normal distribution.
From (\ref{eqn: variance}), and by using the above two estimates, we conclude that
\[\Var[ N_D]=\sum_{|j| < M}   
\Phi\left(\frac{j}{\sigma}\right)\Phi\left(\frac{-j}{\sigma}\right) + o((\log \sigma)^{1/2+2\eps}).\] 
By considering the Riemann sum of $\Phi(x)\Phi(-x)$, we see that
\[\left|\sum_{|j|<M} \Phi\left(\frac{j}{\sigma}\right)\Phi\left(\frac{-j}{\sigma}\right)-
\sigma \int_{-(\log \sigma)^{1/2+\eps}}^{(\log \sigma)^{1/2+\eps}} \Phi(x)\Phi(-x) \,\dd x \right|= O(1).\]
(The difference in the left-hand side of above equation is less than the total variation of the integrand, possibly plus a constant.)

Using the standard bound 
\[\Phi(-x)=\frac{1}{\sqrt{2\pi}}\int_x^{+\infty} e^{-t^2/2}\,\dd t=\frac{1}{\sqrt{2\pi}}\int_{x^2}^{+\infty} \frac{1}{2\sqrt{t}} e^{-t/2}\,\dd t< \frac{1}{\sqrt{2\pi}x}e^{\frac{-x^2}{2}}\]
for any $x>0$, we have
\begin{align*}
\int_{ \{x:|x|>(\log \sigma)^{1/2+\eps}\} } \Phi(x)\Phi(-x) \,\dd x &\leq 
2\int_{(\log \sigma)^{1/2+\eps}}^{+\infty} \Phi(-x) \,\dd x \\ 
&\leq \frac{2}{(\log \sigma)^{1/2+\eps}} \Phi(-(\log \sigma)^{1/2+\eps})\\
&=o(1/\sigma).
\end{align*}
Combining all these facts, we deduce that
\[\Var[ N_D]=\sigma \int_{-\infty}^{+\infty} \Phi(x)\Phi(-x)\,\dd x+ o((\log \sigma)^{1/2+2\eps}).\]
On the other hand, from integration by parts, we have
\[\int_{-\infty}^{+\infty} \Phi(x)\Phi(-x)\,\dd x=\frac{1}{\sqrt{\pi}}\]
and (\ref{appro: variance}) follows.
\end{proof}

\begin{rem} 
Clearly, the restriction of point process $\mathcal{X}^{(n)}$ to a spherical cap $D$ is also a determinantal point process with kernel  $1_{g(D)}(z)K^{(n)}(z,w)1_{g(D)}(w)$. Denote by $\mathscr{K}^{(n)}_{g(D)}$ the integral operator in $L^2(g(D))$ obtained by considering this kernel. Then the distribution of $N_D$ is the sum of independent Bernoulli random variables, whose expectations are the eigenvalues of $\mathscr{K}^{(n)}_{g(D)}$ (see e.g. \cite{AGZ}, Corollary 4.2.24). So from (\ref{dist1: N_D}) we conclude that the non-zero eigenvalues of operator $\mathscr{K}^{(n)}_{g(D)}$ are equal to $\Prob(S_n >k)$, $0 \leq k \leq n-1$, where $S_n$ is a binomial random variable with parameters $n$ and $\frac{|D|}{4\pi}$.
\end{rem}
From Lemma \ref{lem:variance} and the general central limit theorem for determinantal point processes \cite{Soshnikov} (due to Costin and Lebowitz \cite{Costin} in case of the sine kernel) we have the following theorem.
\begin{thm}
Let $D$ be a spherical cap on $\Sphere$ (depending on $n$) such that $\frac{1}{|D|},\frac{1}{|D^c|}=o(n)$. Then
\begin{equation*}
\frac{N_D-\frac{n|D|}{4\pi}}{\frac{1}{2} \pi^{-3/4} n^{1/4} \sqrt[4]{|D||D^c|}} \todist N(0,1)
\end{equation*}
as $n \to +\infty$.
\end{thm}
Next we prove Theorem \ref{thm: Discrepancy}.
\begin{proof}[Proof of Theorem \ref{thm: Discrepancy}]
We use the notation from the proof of Lemma \ref{lem:variance}. As we have seen before
\[\Var[N_D]=\sum_{k=0}^{n-1} \Prob(S_n \leq k) \Prob(S_n>k).\]
Similar to inequality (\ref{estimate: Chernoff}), we have
\begin{align*}
&\sum_{0 \leq k \leq n-1} \Prob(S_n \leq k) \Prob(S_n>k) \leq \\
 &2 \left[ \sum_{0 \leq j \leq \sigma^2} \exp\left(-\frac{j^2}{4\sigma^2}\right)+\sum_{j>\sigma^2 } \exp(-j/4) \right].
\end{align*}
Also, we see
\[\sum_{1 \leq j \leq \sigma^2} \exp\left(-\frac{j^2}{4\sigma^2}\right) \leq \sigma \int_{0}^{+\infty} \exp(-x^2/4)\,\dd x.\]
So we deduce that there exist absolute constants $c_1, c_2$ such that for any spherical cap $D$,
\begin{equation} \label{inequality: variance}
\Var[N_D] \leq c_1\sigma+c_2
\end{equation}
and since $\sigma \leq \frac{\sqrt{n}}{2}$, we have for some absolute constant $c_3$,
\[\Var[N_D] \leq c_3\sqrt{n}.\]

We know that for any spherical cap $D$, the random variable $N_D$ has the same distribution as $\sum_{k} \eta_k$ where $\eta_k$ are independent Bernoulli random variables (see equation (\ref{dist1: N_D})). From Bernstein's inequality we get a concentration estimate for $N_D$:
\[\Prob \left(|N_D-\E N_D|\geq t \right) \leq  2\exp\left(-\min\left(\frac{t^2}{4\Var[N_D]},\frac{t}{4}\right)\right) \quad , \quad t\geq 0.\]
Let $t=O(n^{1/4}\sqrt{\log n})$. Then for any $M>0$ we have 
\begin{equation} \label{eqn: concentration of N_D}
N_D=\E N_D+O( n^{1/4}\sqrt{\log n})
\end{equation}
with probability $1-n^{-M}$ where the implied constant in the $O()$ notation is independent of $D$. It can be easily shown that there is a collection $\mathcal{A}'$ of $n^c$ spherical caps, for some constant $c>0$,  such that for any spherical cap $D$ there exist $D_1,D_2 \in \mathcal{A}'$ having the properties $D_1 \subset D \subset D_2$ and $|D_2\backslash D_1|\leq \frac{1}{n}$. Thus
\[\mathbf{D}(\{P_1,\ldots,P_n\}) \leq 1+\max_{D \in \mathcal{A}'} |N_{D'}-\E N_{D'}|.\]
So from the union bound, we see that the equation (\ref{eqn: concentration of N_D}) is also true uniformly in $D$. This completes the proof of the theorem.    
\end{proof}

\section{{\bf Hole probability and largest empty cap}}
\subsection{Hole probability}
In this subsection, we establish the asymptotic behavior of $\Delta_n(\alpha):=\Prob(N_D=0)$, where $D$ is a spherical cap on $\Sphere$ such that $|D|=4\pi \alpha$. We use the notation from the proof of Lemma \ref{lem:variance}. We know that $\Delta_n(\alpha)$ is equal to the Fredholm determinant of $\Id-\mathscr{K}^{(n)}_{g(D)}$. Also, from (\ref{dist1: N_D}) and (\ref{eqn: Bernoulli}) we have 
\[\Delta_n(\alpha)=\prod_{k=0}^{n-1} \Prob(S_n\leq k).\]
\begin{prop}\label{prop: Hole probability}
There exist a positive constant $c'$ such that
\begin{equation} \label{Prop: Hole probability}
\log\Delta_n(\alpha)=\frac{n^2}{2}(\alpha+\log(1-\alpha))+O(n\alpha\log(n\alpha))
\end{equation}
uniformly in $c'/n<\alpha<1-c'/n.$
\end{prop}

\begin{proof}
We can write
\begin{equation} \label{eqn: Hole probability}
\log\Delta_n(\alpha)=\sum_{k=0}^{n-1} \log \Prob(S_n\leq k).
\end{equation}
By the exponential Chebyshev inequality,  we have 
\begin{equation} \label{Hole probability: Upper bound}
\Prob(S_n\leq k) \leq e^{-n I(\frac{k}{n})}
\end{equation}
for any $k \leq n\alpha$ where $I(x)=\sup_{t \in \RR}(tx- \log(\E e^{t B_1}) )$ is the Legendre transform of the cumulant
generating function of $B_1$. One easily computes that $I(x)=x\log \frac{x}{\alpha}+(1-x)\log \frac{1-x}{1-\alpha}$.
On the other hand, from Stirling's formula we have
\begin{align*}
\Prob(S_n\leq k) \geq \binom{n}{k} \alpha^k(1-\alpha)^{n-k} &\geq c_0 \frac{n^{n+1/2}}{k^{k+1/2} (n-k)^{n-k+1/2}} \alpha ^k(1-\alpha)^{n-k} \\
&\geq c_0 \left(\frac{n}{k(n-k)}\right)^{1/2} e^{-n I(\frac{k}{n})}
\end{align*}
 for $0<k \leq n\alpha$ and some absolute constant $c_0$. Therefore, for all $0<k \leq n\alpha$, we have
\begin{equation} \label{Hole probability: Lower bound}
\log \Prob(S_n\leq k) \geq -nI(k/n)-\frac{1}{2} \log (n\alpha) +\log c_0.
\end{equation}
From (\ref{Hole probability: Upper bound}) and (\ref{Hole probability: Lower bound}), it follows that 
\begin{equation} \label{Hole probability: appro1}
\left|\sum_{0 \leq k \leq n\alpha} \log \Prob(S_n\leq k)+n \sum_{0 \leq k \leq n\alpha} I(k/n) \right|\leq n\alpha\left(\frac{1}{2} \log (n\alpha)-\log c_0\right)
\end{equation}
for sufficiently large $n$ (for $k=0$, note that $\log \Prob(S_n=0)=-nI(0)$).

Since the median of binomial distribution is either $\lfloor n\alpha \rfloor$ or $\lceil n\alpha \rceil$, this implies that $\Prob(S_n>k) \leq 1/2$ when $k>n\alpha$. Using the bound $-\log (1-x)<2x$, $0 \leq x \leq 1/2$, we then obtain   
\[\sum_{n\alpha<k \leq n-1} -\log \Prob(S_n\leq k) < 2\sum_{n\alpha<k \leq n-1} \Prob(S_n>k).\] 
From (\ref{eqn: variance}), the right hand side of above inequality is smaller than $4 \Var[N_D]$, and then from (\ref{inequality: variance}) we have
\begin{equation} \label{Hole probability: appro2}
\sum_{n\alpha<k \leq n-1} -\log \Prob(S_n\leq k) \leq c_1\sqrt{n\alpha(1-\alpha)}+c_2
\end{equation}
for some constants $c_1,c_2$.
Also, by Riemann integration, we have
\begin{equation} \label{Hole probability: appro3}
\left|\sum_{0 \leq k \leq n\alpha} I(k/n)-n\int_0^\alpha I(x) \,\dd x \right| \leq I(0)-I(\alpha)=-\log(1-\alpha)
\end{equation}
(Since $I$ is an increasing function on $[0,\alpha]$ and $I(\alpha)=0$)
From (\ref{eqn: Hole probability}), (\ref{Hole probability: appro1}), (\ref{Hole probability: appro2}) and (\ref{Hole probability: appro3}), we have
\begin{align*}
&\left|\log\Delta_n(\alpha)+n^2\int_0^\alpha I(x) \,\dd x \right| \leq \\
&\frac{1}{2} n\alpha \log(n\alpha)-n\log(1-\alpha)-n\alpha \log c_0+c_1\sqrt{n\alpha(1-\alpha)}+c_2. 
\end{align*}
So we conclude that for sufficiently large constant $c'$
\[\left|\log\Delta_n(\alpha)-\frac{n^2}{2}(\alpha+\log(1-\alpha))\right|=O(n\alpha\log(n\alpha))\]
uniformly in $c'/n<\alpha<1-c'/n.$
\end{proof}
\subsection{largest empty cap}

\begin{proof}[Proof of Theorem \ref{thm: Largest empty cap}]
We follow the proof of  Theorem 1.3 in \cite{Arous}.
Let $\eps>0$ be arbitrary and set 
\[X_n:=\frac{n}{8\pi \sqrt{\log n}}M_n\]
we easily see the inequality
\begin{equation*}
\E (|X_n-1|^s) \leq \epsilon^s+\Prob (X_n<1-\epsilon)+
\E (|X_n-1|^s 1_{\{X_n>1+\epsilon\}}). 
\end{equation*}
Hence, it suffices to show that the last two terms go to zero as $n$ tends to infinity. Integrating by parts, we get
\begin{equation*}
\E(|X_n-1|^s \mathrm{I}_{\{X_n>1+\epsilon\}})
=\int_{1+\epsilon}^{+\infty} s(u-1)^{s-1}\Prob(X_n>u)\,\mathrm{d}u+\epsilon^s \Prob (X_n>1+\epsilon).
\end{equation*}
Next, we give an upper bound for $\Prob(X_n>u)$. We can choose on $\Sphere$ at most $4n$ points $q_1,\ldots,q_m$ so that every spherical cap with area $4\pi/n$ contains at least one of these points. (We can find $m \leq 4n$ spherical caps with area $\pi/n$ such that there exists no spherical cap with this area that does not intersect these $m$ spherical caps. Let $q_1,\ldots,q_m$ be the center of these $m$ spherical caps. One can check that $q_1,\ldots,q_m$ have desired properties.)  Let $Y_i$, $1 \leq i \leq m$, be the area of the largest empty cap centred at $q_i$. Note that there exists some $q_j$ within distance $2/\sqrt{n}$ of the center of the largest empty cap. From this, we can conclude that for any $u \geq 1+\eps$ and sufficiently large $n$,
\begin{align} \label{appro3}
\Prob (X_n>u)&\leq \Prob\left(M_n>\frac{8\pi \sqrt{\log n}}{n} (1+\eps)\right) \\ 
&\leq \sum_{i=1}^{m} \Prob \left(Y_i> \frac{8\pi\sqrt{\log n}}{n} (1+\eps /2) \right). \nonumber
\end{align}
Thus, for sufficiently large $n$
\begin{equation} \label{appro2}
\mathbf{P}(X_n>u)\leq 4n\Delta_n\left( \frac{2\sqrt{\log n}}{n} (1+\eps /2)\right).
\end{equation}
Using (\ref{Prop: Hole probability}) and (\ref{appro2}), we conclude that there is some $\delta>0$ such that
\begin{equation} \label{appro1}
\Prob (X_n >  u)=o(n^{-\delta})
\end{equation}
 uniformly for $u \geq 1+\eps$. Also, we can write
 \begin{align*} 
\int_{1+\eps}^{+\infty} s(u-1)^{s-1}\mathbf{P}(X_n>u)\,\mathrm{d}u=&
\int_{1+\eps}^{\log n} s(u-1)^{s-1}\Prob(X_n>u)\,\mathrm{d}u\\+
&\int_{\log n}^{\frac{n}{2\sqrt{\log n}}} s(u-1)^{s-1}\Prob(X_n>u)\,\mathrm{d}u. 
\end{align*}
 The first integral on the right hand side goes to zero as $n \to +\infty$, thanks to (\ref{appro1}). 
 Using a similar argument as in (\ref{appro3}) and the crude upper bound  
\[\Delta_n(\alpha)\leq \Prob(S_n=0)= (1-\alpha)^n\] we conclude that for any fixed $C>0$, one has 
\[\Prob(X_n >u)=o(n^{-C})\]
uniformly for $u \geq \log n$.
 Thus the second integral also goes to zero as $n \to +\infty$.

To prove Theorem \ref{thm: Largest empty cap}, it suffices to show that $\Prob(X_n<1-\eps)$  converges to zero. 
The following lemma is similar to Lemma 3.3 in \cite{Arous}. 
 The proof is based on negative association property for the events such as $\{N_{D_1}=0\}$ and $\{N_{D_2}=0\}$ where $D_1,D_2$ are two disjoint spherical caps on $\Sphere$. Note that negative association property holds in this case. See Theorem 1.4 in \cite{Gosh} or proof of Lemma 3.8 in \cite{Arous}. 
\begin{lem} 
Consider a set of disjoint spherical caps on the sphere. Let $F_n$ be the number of such caps free of eigenvalues. Then $\Var (F_n) \leq \E (F_n).$
\end{lem}
 Now consider $c_n=\Omega\left(n/ \sqrt{\log n}\right)$ disjoint spherical caps $D_1, D_2, \ldots, D_{c_n}$ with area $\frac{(1-\epsilon)8\pi \sqrt{\log n}}{n}$. If $F_n$ be the number of such caps free of eigenvalues then from previous lemma and Chebyshev's inequality one has
\[\Prob(X_n<1-\epsilon)\leq \Prob(F_n=0)\leq \frac{\Var(F_n)}{\E (F_n)^2} \leq \frac{1}{\E (F_n)}\]
Also, using (\ref{Prop: Hole probability}) we have 
 \[\E (F_n)=n^{1-(1-\epsilon)^2+o(1)},\]
 which implies that $\Prob(X_n<1-\epsilon)\to 0$ as desired.
 \end{proof}
\subsection{Nearest neighbour statistics}
Consider the random point process $\mathcal{X}^{(n)}=\sum_{j=1}^n \delta_{P_j}$. Define for $1 \leq j \leq n$
\[d_j:=\min_{i \neq j} |P_i-P_j|\]
the minimum distance from $P_j$ to the remaining points. We define,  as in \cite{Sarnak}, the nearest neighbour spacing measure $\mu(P_1,\ldots,P_n)$ on $[0,+\infty)$ by
\begin{equation} \label{Def: Nearest neighbour}
\mu(P_1,\ldots,P_n):=\frac{1}{n} \sum_{j=1}^{n} \delta_{\frac{n}{4}d_j^2}.
\end{equation}
Let
\[Q(x)=-\frac{\dd}{\dd x} E_{\infty}(x)\]
where
\[E_{\infty}(x)=\lim_{n \to+\infty} E_n(x) \quad ,\quad E_n(x)=\prod_{k=1}^{n-1} e^{-x} \sum_{j=0}^{k} \frac{x^j}{j!}.\]
We want to show that,
\begin{thm} \label{Theorem: Nearest neighbour}
As $n \to +\infty$,
\begin{equation} \label{thm: Nearest neighbour}
\mu(P_1,\ldots,P_n) \longrightarrow Q(x) \,\dd x. 
\end{equation}
\end{thm}

One can easily check that for independent uniform points on sphere, this measure converges to $e^{-x} \, \dd x$ as $n$ tends to infinity. See Figure \ref{Fig: Random vs spherical ensemble}.

\begin{proof}[Proof of Theorem \ref{Theorem: Nearest neighbour}]
For fixed $x>0$ let
\[Y_n=\sum_{j=1}^{n} \delta_{\frac{n}{4}d_j^2} (0,x).\]
To prove (\ref{thm: Nearest neighbour}) it suffices to show that
\begin{equation} \label{Nearest neighbour: a.s.}
\frac{Y_n}{n} \overset{\mbox{\rm \scriptsize a.s.}}{\longrightarrow} 1-E_{\infty}(x).
\end{equation}
First, we compute the expectation of $Y_n$. 
 Let $\tilde{\Delta}_n(\alpha)$ denote the probability that  the spherical cap $D$ with area $4\pi \alpha$ centred at $P_j$ 
has no other points of $\{P_1,\ldots,P_n\}$ in its interior. With this definition we obtain
\begin{equation} \label{eqn: expectation of X_n}
\E Y_n=n(1-\tilde{\Delta}_n(x/n)).
\end{equation}
We first show that
\begin{equation} \label{eqn:  Nearest neighbour probability}
\tilde{\Delta}_n(\alpha)=\frac{1}{(1-\alpha)^n} \Delta_n(\alpha)=\prod_{k=1}^{n-1} \Prob(S_n \leq k)
\end{equation}
where $S_n$ is a binomial random variable with parameters $n$ and $\alpha$.
For $\alpha'<\alpha$ define $\Delta_n(\alpha,\alpha'):=\Prob\left(N_{D\backslash D'}=0\right)$ where $|D|=4\pi \alpha$, $|D'|=4\pi \alpha'$ and $D,D'$ have the same center. 
$\Delta_n(\alpha,\alpha')-\Delta_n(\alpha)$ is the probability that there are no points of $\mathcal{X}^{(n)}$ in the $D\backslash D'$, but there is at least one point (of $\mathcal{X}^{(n)}$) in the $D'$. 
Now, conditioning on the $\{N_{D'} \geq 1\}$ and letting $\alpha' \to 0$ (keeping only the first-order terms), one obtains 
\[\frac{\dd}{\dd \alpha'}\Delta_n(\alpha,\alpha')_{|_{\alpha'=0}}=n \tilde{\Delta}_n(\alpha)\]
or equivalently that
\begin{equation} \label{eqn: Nearest neighbour probability 1}
\tilde{\Delta}_n(\alpha)=\frac{1}{n}\frac{\dd}{\dd \alpha'} \log \Delta_n(\alpha,\alpha')_{|_{\alpha'=0}} \Delta_n(\alpha).
\end{equation}
Similar to (\ref{dist0: N_D}), we have
\[N_{D\backslash D'} \dist \Big|\{k: 0 \leq k \leq n-1 \, , \, \frac{\alpha'}{1-\alpha'}<Q_k<\frac{\alpha}{1-\alpha} \}\Big|.\]
where $Q_k$, $0 \leq k \leq n-1$, are as in Lemma \ref{lem:variance}. A straightforward computation yields
\begin{align*}
&\left. \frac{\dd}{\dd \alpha'} \log \Delta_n(\alpha,\alpha')\right| _{\alpha'=0}=\\
&\left. \frac{\dd}{\dd \alpha'} \sum_{k=0}^{n-1} \log \left( \int_{0}^{\frac{\alpha'}{1-\alpha'}} \frac{n\binom{n-1}{k}q^k}{(1+q)^{n+1}}\,\mathrm{d}q + \int_{\frac{\alpha}{1-\alpha}}^{+\infty} \frac{n\binom{n-1}{k}q^k}{(1+q)^{n+1}}\,\mathrm{d}q \right)\right|_{\alpha'=0}=\frac{n}{(1-\alpha)^n}.
\end{align*}
Thus, inserting this into (\ref{eqn: Nearest neighbour probability 1}), we obtain (\ref{eqn:  Nearest neighbour probability}). 

Next, we show that
\begin{equation} \label{appro: Nearest neighbour probability}
|\tilde{\Delta}_n(x/n)-E_n(x)|=O\left(\frac{\log n}{n}\right).
\end{equation}
For this, we will use identity (\ref{eqn:  Nearest neighbour probability}) and Poisson approximation.
From a result in \cite{Cam} we have
\begin{equation} \label{appro: Poisson approximation}
\left| \Prob(S_n \leq k)-e^{-x} \sum_{j=0}^{k} \frac{x^j}{j!} \right|=O(n^{-1})
\end{equation}
uniformly for all $k$.
From (\ref{inequality: Bernstein}) we see that there exists $M=O(\log n)$ such that for every $k \geq M$ we have
\begin{equation} \label{appro: Poisson approximation2}
\Prob(S_n \leq k) \geq 1-\frac{1}{n^2} \quad , \quad e^{-x} \sum_{j=0}^{k} \frac{x^j}{j!} \geq 1-\frac{1}{n^2}
\end{equation}
where, for the second inequality, note that from Taylor's theorem  
\[\left| 1-e^{-x} \sum_{j=0}^{k} \frac{x^j}{j!}\right| \leq e^{-x} \left| e^x- \sum_{j=0}^{k} \frac{x^j}{j!}\right| \leq \frac{x^{k+1}}{(k+1)!} \] 
and then use Stirling's formula. Therefore, using the inequality
\begin{align*}
|\tilde{\Delta}_n(x/n)-E_n(x)|&=\left|\prod_{k=1}^{n-1} \Prob(S_n \leq k)-\prod_{k=1}^{n-1} e^{-x} \sum_{j=0}^{k} \frac{x^j}{j!}\right| \\ 
& \leq \sum_{k=1}^{n-1} \left| \Prob(S_n \leq k)- e^{-x} \sum_{j=0}^{k} \frac{x^j}{j!}\right|, 
\end{align*}
\ref{appro: Poisson approximation} and \ref{appro: Poisson approximation2} we get (\ref{appro: Nearest neighbour probability}). From (\ref{eqn: expectation of X_n}) and (\ref{appro: Nearest neighbour probability}) we deduce that
\[\lim_{n \to +\infty} \frac{\E Y_n}{n}=1-E_{\infty}(x).\]
In \cite{Peres} (See Theorem 3.5 therein) it has been shown that for every 1-Lipschitz function on finite counting measures on $\Sphere$ such as $f$, (This means that deleting or adding one point to a configuration of points on $\Sphere$ changes $f$ by at most 1. The point process on $\Sphere$ can also be viewed as a random counting measure on $\Sphere$.) satisfies the following concentration inequality: 
\[\Prob(|f-\E f| \geq a) \leq 5\exp\left(-\frac{a^2}{16(a+2n)}\right).\]
Fix $r>0$ and let $f$ count the number of points of point process $\{P_1,\ldots,P_n\}$ such that $d_j<r$. It is simple to check that $f$ is Lipschitz with some constant $c'>0$. (If one point is added then $f$ increases by at most $c'$. Using a simple geometric argument, one can choose $c' = 7$.) Applying previous inequality to $f/c'$, we conclude that   
\[\Prob(|Y_n-\E Y_n| \geq a) \leq 5\exp\left(-\frac{a^2}{16c'(a+2c'n)}\right).\]
Using the Borel-Cantelli Lemma, this gives (\ref{Nearest neighbour: a.s.}) as desired. 
\end{proof}
\begin{figure}[h]
\begin{center}
 \includegraphics[width=110mm]
{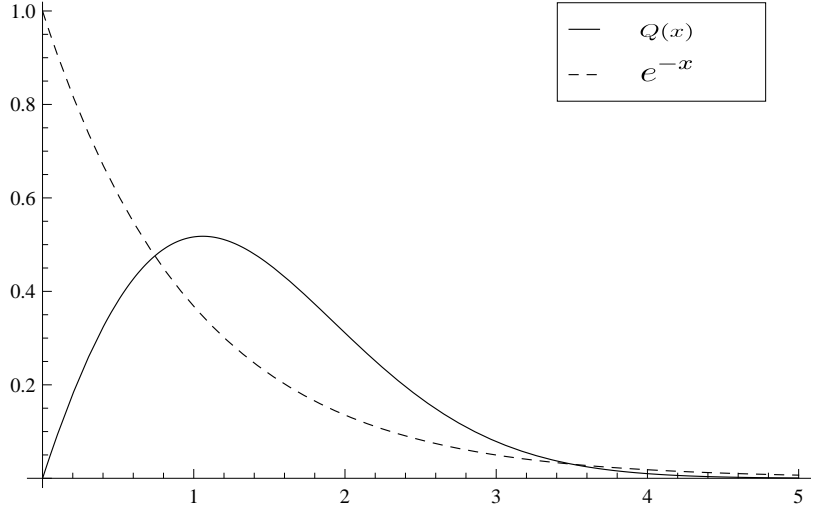}
 \caption{Comparison between the density of nearest neighbour spacing measure $\mu$ (see (\ref{Def: Nearest neighbour})) in the case that the points are independently chosen as uniform distribution (on sphere)  and spherical ensemble case in the limit $n \to +\infty$.} \label{Fig: Random vs spherical ensemble}
\end{center}
\end{figure}
\begin{rem}
The infinite Ginibre ensemble is a translation invariant determinantal point process on the complex plane $\CC$ with kernel $K(z,w)=e^{z\bar{w}}$ with respect to the Gaussian measure $\frac{1}{\pi} e^{-|z|^2}\,\dd z$. It can be viewed as the local limit of the law  of eigenvalues of random matrices from the complex Ginibre ensemble (see \cite{GAF-Book} for more details). 
From (\ref{appro: Nearest neighbour probability}) we deduce that
\[\lim_{n \to +\infty} \Delta_n(x/n)=\exp(-x) E_{\infty}(x).\]
The right-hand side of above equation is equal to the probability that a disk of radius $\sqrt{x}$ in the complex plane contains no points of infinite Ginibre ensemble (see e.g. \cite{Forrester2}). Also, if we consider the complex Ginibre ensemble, then $E_n(x)$ is the conditional probability that if one eigenvalue (of an $n \times n$ random matrix with i.i.d. standard complex Gaussian entries) lies at the origin, all $n-1$ others are further away than $\sqrt{x}$. Compare with the definition of $\tilde{\Delta}$ and see (\ref{appro: Nearest neighbour probability}).
\end{rem}

\section{\bf{ Riesz and logarithmic energy}}
The purpose of this section is to establish Theorem \ref{theorem: Riesz energy} and Corollary \ref{Cor: Riesz energy}.
We start with computing the correlation functions of $\mathcal{X}^{(n)}$ on $\Sphere$. Let $\mathring{\rho}_k^{(n)}(p_1, \ldots, p_k)$, $k \geq 1,$ be the correlation functions of the point process $\mathcal{X}^{(n)}$ with respect to the surface area measure $\dd \nu(p)$.   
Set $g(p)=z$ and $g(q)=w$. Since $\dd \nu(p)=\frac{4}{(1+|z|^2)^2} \,\dd z$ we conclude that
\[\mathring{\rho}_{k}^{(n)}(p_1,\ldots , p_k)=\det(\mathring{K}^{(n)}(p_i,p_j))_{1 \leq i,j \leq k}\]
where 
\[\mathring{K}^{(n)}(p,q)=\frac{n}{4\pi} \frac{K^{(n)}(z,w)}{(1+|z|^2)^{\frac{n-1}{2}}(1+|w|^2)^{\frac{n-1}{2}}}.\]
Also, one can easily verify that
\[|p-q|=\frac{2|z-w|}{\sqrt{1+|z|^2}\sqrt{1+|w|^2}}.\]
Recall that $K^{(n)}(z,w)=(1+z\bar{w})^{n-1}$. A short computation then shows that
\begin{align} \label{eqn: norm of kernel}
|\mathring{K}^{(n)}(p,q)|^2 &=\left(\frac{n}{4 \pi}\right)^{2} \left(\frac{1+|zw|^2+2\Re(z\bar{w})}{(1+|z|^2)(1+|w|^2)}\right)^{n-1} \\
&=\left(\frac{n}{4 \pi}\right)^{2} \left(1-\frac{|z-w|^2}{(1+|z|^2)(1+|w|^2)}\right)^{n-1} \nonumber \\
&= \left(\frac{n}{4 \pi}\right)^{2}\left(1-\frac{|p-q|^2}{4}\right)^{n-1}. \nonumber
\end{align}
Thus from above equation, the 2-point correlation function $\mathring{\rho}_2^{(n)}$ is given by
\begin{equation} \label{eqn: Rho2} 
\mathring{\rho}_{2}^{(n)}(p,q)=
\left|
\begin{array}{ll}
\mathring{K}^{(n)}(p,p)&\mathring{K}^{(n)}(p,q)\\ [1mm]
\mathring{K}^{(n)}(q,p)&\mathring{K}^{(n)}(q,q)
\end{array} 
\right|
=\left(\frac{n}{4 \pi}\right)^{2}\left(1-\left(1-\frac{|p-q|^2}{4}\right)^{n-1}\right).
\end{equation}

\begin{proof}[Proof of Theorem \ref{theorem: Riesz energy}]
We begin with part ii). Similar to (\ref{eqn: def of rho}), we have for a suitable test function $F$ 
\begin{equation} \label{expectation and pho2}
\E \sum_{i \neq j} F(P_i,P_j)= \int_{\Sphere \times \Sphere} F(p,q) \mathring{\rho}_2^{(n)}(p,q) \,\dd \nu(p) \,\dd \nu(q). 
\end{equation}
Thus by setting $F(p,q)=|p-q|^{-s}$ we obtain
\[\E E_{s}(P_1,\ldots,P_n)=\int_{\Sphere \times \Sphere} |p-q|^{-s} \mathring{\rho}_2^{(n)}(p,q) \,\dd \nu(p) \,\dd \nu(q).\]
 Notice that the point process is invariant in distribution under isometries of the sphere and so by Fubini's theorem we conclude that
\[\E E_{s}(P_1,\ldots,P_n)=4\pi \int_{\Sphere} |p-q|^{-s} \mathring{\rho}_2^{(n)}(p,q) \,\dd \nu(p)\]
where $q=g^{-1}(0).$ Thus we can write
\[\E E_{s}(P_1,\ldots,P_n)=\frac{n^2}{2^s \pi} \int_{\CC} |z|^{-s} \frac{(1+|z|^2)^{n-1}-1}{(1+|z|^2)^{n+1-s/2}} \,\dd z.\]
(This can also be obtained directly from Equation \ref{eqn: def of rho}, letting 
\[F(z,w)=\left(\frac{2|z-w|}{\sqrt{1+|z|^2}\sqrt{1+|w|^2}}\right)^{-s}\] and using suitable linear fractional transformations corresponding to the rotations of $\Sphere$ by stereographic projection). Changing to polar coordinates, we get
\begin{equation} \label{eqn: Riesz energy 1}
\E E_{s}(P_1,\ldots,P_n)=\frac{n^2}{2^{s-1}} \int_{0}^{+\infty} \frac{1}{r^{s-1}}\frac{(1+r^2)^{n-1}-1}{(1+r^2)^{n+1-s/2}} \,\dd r.
\end{equation} 
 Making the change of variable $u=r^2$, we see that 
\begin{align} \label{eqn: Riesz energy 2}
\E E_{s}(P_1,\ldots,P_n) &=\frac{n^2}{2^{s}} \sum_{j=1}^{n-1} \binom{n-1}{j} \int_{0}^{+\infty} \frac{1}{u^{s/2-j}(1+u)^{n+1-s/2}}\,\dd u \\ 
&=\frac{n^2}{2^{s}} \sum_{j=1}^{n-1} \frac{\Gamma(n) \Gamma(j+1-s/2)}{\Gamma(n+1-s/2)\Gamma(j+1)}. \nonumber
\end{align}
In the last line, we have used the beta function identity 
\begin{align*}
\int_{0}^{+\infty} \frac{1}{u^{s/2-j}(1+u)^{n+1-s/2}}\,\dd u &=\int_{0}^{1} t^{j-s/2}(1-t)^{n-j-1}\,\dd t \\ 
&=\frac{\Gamma(j+1-s/2)\Gamma(n-j)}{\Gamma(n+1-s/2)}.  
\end{align*}
By using induction on $n$, one can show that for $s \neq 2$ 
\[\sum_{j=0}^{n-1} \frac{\Gamma(j+1-s/2)}{\Gamma(j+1)}=\frac{\Gamma(n+1-s/2)}{(1-s/2)\Gamma(n)}.\]
Combining this with (\ref{eqn: Riesz energy 2}), we obtain (\ref{thm: Energy part 2}) as required.

For $s=2$, from (\ref{eqn: Riesz energy 2}) we have   
\[\E E_{2}(P_1,\ldots , P_n)=\frac{n^2}{4}\sum_{j=1}^{n-1} \frac{1}{j}.\]
Also, from Euler-Maclaurin Summation Formula, we know 
\begin{equation} \label{appro: Harmonic}
\sum_{j=1}^{n} \frac{1}{j}=\log n+ \gamma +\frac{1}{2n}-\frac{1}{12n^2}+O\left(\frac{1}{n^4}\right),
\end{equation}
(see the proof of Theorem 3 in \cite{Boas}) and hence
\[\E E_{2}(P_1,\ldots , P_n)=\frac{1}{4}n^2 \log n+\frac{\gamma}{4}n^2-\frac{n}{8}-\frac{1}{48}+O\left(\frac{1}{n^2}\right).\]
It remains to prove (\ref{thm: Energy part 1}). By differentiation right hand side of (\ref{eqn: Riesz energy 1}) with respect to $s$ at $s=0$, we conclude that
\begin{align*}
 \E E_{\log}(P_1,\ldots,P_n) &\left.=\frac{\dd}{\dd s} \E E_{s}(P_1,\ldots,P_n) \right|_{s=0^+} \\
 &=\left. \frac{\dd}{\dd s} \left[ \frac{2^{1-s}}{2-s}n^2-\frac{\Gamma(n)\Gamma(1-s/2)}{2^{s}\Gamma(n+1-s/2)}n^2 \right] \right|_{s=0^+}.
\end{align*}
It is well known that
\[\Gamma'(n)=\Gamma(n) \left(- \gamma+\sum_{j=1}^{n-1} \frac{1}{j}\right)\]
(see, e.g.  6.3.1-2 of \cite{Abramowitz}).
Using (\ref{appro: Harmonic}) and above equation, we obtain the claim.
\end{proof}

\begin{proof}[Proof of Corollary \ref{Cor: Riesz energy}]
For $0<s<2$, from Theorem \ref{theorem: Riesz energy}, we conclude that there exist $n$-point set $\{x_1,\ldots,x_n\}$ such that
\[E_s(x_1,\ldots,x_n)\leq \frac{2^{1-s}}{2-s}n^2-\frac{\Gamma(n)\Gamma(1-s/2)}{2^{s}\Gamma(n+1-s/2)}n^2.\] 
Thus, by definition
\[\mathcal{E}_{s}(n) - \frac{2^{1-s}}{2-s}n^2 \leq-\frac{\Gamma(n)\Gamma(1-s/2)}{2^{s}\Gamma(n+1-s/2)}n^2.\] 
It is well known that
\[\lim_{n \to +\infty} \frac{\Gamma(n) n^{1-s/2}}{\Gamma(n+1-s/2)}=1.\]
Also, For $0<\beta<1$ and $x>0$, we have (See \cite{Gamma})
\begin{equation} \label{inequality: Wendel}
\frac{x}{(x+\beta)^{1-\beta}}\leq \frac{\Gamma(x+\beta)}{\Gamma(x)}\leq x^\beta.
\end{equation}
Hence, for $0<s<2$,
 \[\frac{\Gamma(n+1-s/2)}{\Gamma(n)} \leq n^{1-s/2}.\]
Consequently, we get 
\[\mathcal{E}_{s}(n) - \frac{2^{1-s}}{2-s}n^2 \leq-\frac{\Gamma(1-s/2)}{2^{s}}n^{1+s/2}.\] 
if we set $\beta=s/2$ and $x=1-s/2$ in (\ref{inequality: Wendel}), we get for $0<s<2$, $\Gamma(1-s/2)>1$. Thus, for $0<s<2$, we have
\[\frac{\Gamma(1-s/2)}{2^s}>(2\sqrt{2\pi})^{-s}\]
and this shows that the bound (\ref{Energy: new upper bound}) is better than (\ref{Energy: upper bound}). See Figure \ref{Fig: Energy}.
\\
For $-2<s<0$, from (\ref{inequality: Wendel}), we have
\[\frac{\Gamma(n+1-s/2)}{\Gamma(n)}=(n-s/2)\frac{\Gamma(n-s/2)}{\Gamma(n)} \geq n(n-s/2)^{-s/2}\geq n^{1-s/2}.\]
Therefore, using similar argument as above, we get (\ref{Energy: new lower bound}).
If we set $\beta=-s/2$ and $x=1$ in (\ref{inequality: Wendel}), we then have for $-2<s<0$, $\Gamma(1-s/2)<1$, therefore
\[\frac{\Gamma(1-s/2)}{2^s}<(2\sqrt{2\pi})^{-s}.\]
So the bound (\ref{Energy: new lower bound}) is better than (\ref{Energy: lower bound}). See Figure \ref{Fig: Energy}.

From (\ref{thm: Energy part 3}), we easily obtain (\ref{Energy: new upper bound for s=2}). To control the error term $O(1/n^2)$ in (\ref{thm: Energy part 3}) see the proof of Theorem 3 in \cite{Boas}. This term is positive and bounded by $\frac{1}{480 n^2}$. 
\end{proof}
\begin{figure}[h]
\begin{center}
 \includegraphics[width=110mm]
{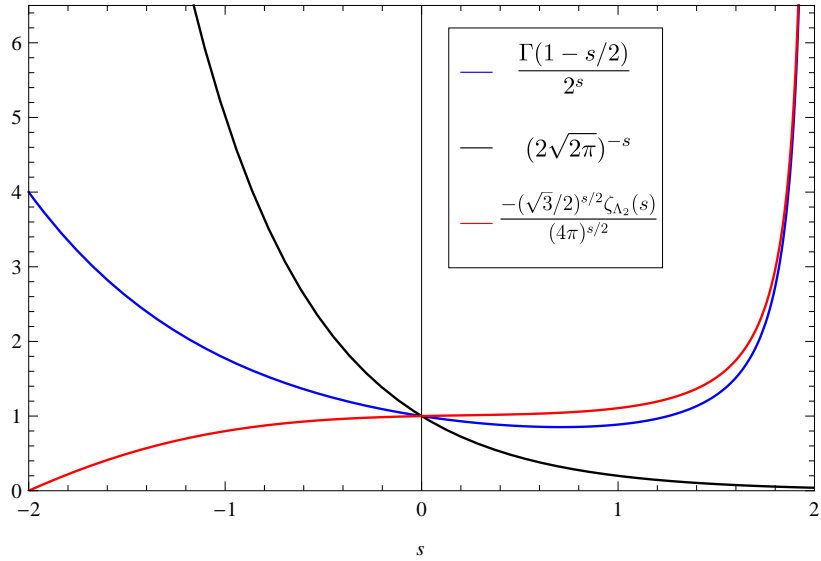}
 \caption{Blue curve: The bound for $(n^2-\mathcal{E}_s(n))/n^{1+s/2}$ given by Corollary \ref{Cor: Riesz energy}. Black curve: The asymptotic bound for $(n^2-\mathcal{E}_s(n))/n^{1+s/2}$ given by \ref{Energy: upper bound} and \ref{Energy: lower bound}. Red curve: conjectured value of $(n^2-\mathcal{E}_s(n))/n^{1+s/2}$ as $n \to +\infty$.} \label{Fig: Energy}
\end{center}
\end{figure}
\begin{rem}
For $s \geq 4$, the expected value of Riesz $s$-energy for $\mathcal{X}^{(n)}$ is infinite. Compare with the points that are chosen randomly and independently on the sphere, with the uniform distribution. in this case,  the expected value of Riesz $s$-energy is infinite for $s \geq 2$. 
\end{rem}

\begin{rem}
Another interesting random point process on $\Sphere$ is the roots of random polynomials via the stereographic projection. Let $f(z)=\sum_{j=0}^n a_j z^j$ where the coefficients $a_j$ are independent complex Gaussian random variables with mean 0 and variance $\binom{n}{j}$. Let $z_1, \ldots, z_n$ are the complex zeros of $f(z)$. In \cite{Shub2011} the expectation of the logarithmic energy for point process $\{g^{-1}(z_1),\ldots,g^{-1}(z_n)\}$ was computed. It was shown that
\[\E E_{\log} (g^{-1}(z_1),\ldots,g^{-1}(z_n))=\left(\frac{1}{2}-\log 2\right)n^2-\frac{1}{2} n\log n-\left(\frac{1}{2}-\log 2\right)n.\]
Compare with (\ref{thm: Energy part 1}).
\end{rem}

\begin{rem}
 Another criterion for the quality of the distribution of points on the sphere is the spherical cap $L^2$-discrepancy, which is given by
\[\Disc_2(\{x_1,\ldots,x_n\})=\left( \int_{-1}^1  \frac{1}{4\pi}  \int_{\Sphere} \left| \sum_{j=1}^n 1_{D(p,t)} (x_j)-\frac{n|D(p,t)|}{4 \pi} \right|^2 \, \dd \nu(p)  \,\dd t\right)^{1/2}.\]
Here $D(p,t)=\{q \in \Sphere: \langle p,q \rangle \leq t \}$.
Stolarsky's  invariance theorem says that (see \cite{Stolarsky})
\[\frac{1}{2} \sum_{i,j=1}^n |x_i-x_j| +\Disc^2_2(\{x_1,\ldots,x_n\})=\frac{2}{3} n^2\]
from (\ref{thm: Energy part 2}) we have  
\[\E \Disc^2_2(P_1,\ldots,P_n)=\frac{\Gamma(3/2)\Gamma(n)n^2}{\Gamma(n+3/2)} \leq \Gamma(3/2) \sqrt{n}.\]
which implies, for $n \geq 2$,
\[\E \Disc_2(P_1,\ldots,P_n) \leq \sqrt{\Gamma(3/2)} \sqrt[4]{n}.\]
\end{rem}

\section{{\bf Minimum spacing}}
\subsubsection*{Point pair statistics} 
Recall the definition (\ref{def: point pair statistics}) of the function $G_{t,n}$. We first compute the expectation of $G_{t,n}$. The argument is similar to what was done in the proof of Theorem \ref{theorem: Riesz energy} in Section 4.  Let $F(p,q)=1_{\{|p-q|\leq t\}}$. From (\ref{expectation and pho2}) and Fubini's theorem we have
\begin{equation} \label{eqn: point pair statistics}
2\E G_{t,n}=4\pi \int_{\Sphere} F(p,q) \mathring{\rho}_2^{(n)}(p,q) \,\dd \nu(p).
\end{equation}
Thus
\begin{align*}
\E G_{t,n}&=n^2 \int_{0}^{\frac{t}{\sqrt{4-t^2}}} r\left(\frac{1}{(1+r^2)^2}-\frac{1}{(1+r^2)^{n+1}}\right) \,\dd r \\
&=\frac{n^2 t^2}{8}-\frac{n}{2}\left(1-(1-t^2/4)^n\right).
\end{align*}
 If we set $t=o(\frac{1}{\sqrt{n}})$ then we see that
 \[\E G_{t,n}=\frac{n^3t^4}{64}(1+o(1))\]
and for $t=\frac{x}{n^{3/4}}$, where $x>0$ is fixed, we get
\begin{equation} \label{expectation of point pair statistics}
\lim_{n \to +\infty} \E G_{t,n}=\frac{x^4}{64}.
\end{equation}
The above equation shows that the correct scaling for the minimum spacing is $n^{-3/4}$. To prove Theorem \ref{thm: Minimum spacing}, similar to \cite{Arous}, we will use a modification of the method from Soshnikov. This method has been used in \cite{Arous} and \cite{Soshnikov2} (see also \cite{Soshnikov3}) to solve similar problems in one-dimensional case. We will modify this method that works in our case. Moreover this modification makes the proof simpler.  

The following two lemmas will be used frequently in the proof.
\begin{lem}
For $p,q \in \Sphere$ such that $|p-q|=O(n^{-3/4})$ we have
\begin{equation} \label{appro: Rho2}
\mathring{\rho}_{2}^{(n)}(p,q)=O(n^{3/2}).
\end{equation}
\end{lem}

\begin{proof}
The proof is immediate from equation (\ref{eqn: Rho2}).
\end{proof}
The next lemma will be used to control the $k$-point correlation function in terms of the lower order correlation functions. For the proof see \cite{Hadamard}.
\begin{lem} (Hadamard-Fischer inequality)
Let $M$ be an $n \times n$ (Hermitian) positive definite matrix and let $\omega \subset \{1,2,\ldots,n\}$ be an index set. Let $M_{\omega}$ be the submatrix of $M$ using rows and columns numbered in $\omega$. Then
\[\det(M) \leq \det(M_w) \det(M_{{\bar{\omega}}})\]
where $\bar{\omega}=\{1,2,\ldots,n\}-\omega$.
\end{lem}

Fix $c>0$ and let $t=\frac{c}{n^{3/4}}$. Define for $p \in \Sphere$
\[C^{(n)}(p)=\{q \in \Sphere : |p-q| \leq t \}\]
Consider the random point process
\[\mathcal{X}^{(n)}=\sum_{i=1}^{n} \delta_{P_i}.\]
 We define a new point process $\modifiedX^{(n)}$.
First consider all pairs $(P_i,P_j)$ such that $n^{3/4}\left|P_i-P_j\right|<x$ and 
\[C^{(n)}(P_i)=\{P_j\} \quad , \quad C^{(n)}(P_j)=\{P_i\}.\]
Then from each pair select independently with probability $\frac{1}{2}$ one of the two items, and consider all this points as $\modifiedX^{(n)}$. (Compare this with the modified processes that have been used in \cite{Arous, Soshnikov2} in similar cases.)  Then let $Z_n=\modifiedX^{(n)}(\Sphere)$.
\begin{lem} \label{lem: 3 close points}
\[G_{t,n}-Z_n \todist 0.\]
\end{lem}
\begin{proof}
We show that $G_{t,n} - Z_n \overset{\mbox{\rm \scriptsize $\L^1$}}{\longrightarrow} 0.$ First note that we have
\begin{equation} \label{inequality: pairs and triples}
\E |G_{t,n} - Z_n| \leq \int_{\Sphere} \int_{C^{(n)}(p) \times C^{(n)}(p)} \mathring{\rho}_{3}^{(n)} (p,q_1,q_2) \,\dd \nu(q_1) \,\dd \nu(q_2) \,\dd \nu(p). 
\end{equation}
To see this, observe that $G_{t,n}-Z_n$ is equal to number of pairs $(P_i,P_j)$ such that $0<|P_i-P_j|<t$ and there exist some point $P_k$ where $0<|P_i-P_k|<t$ or $0<|P_j-P_k|<t$. By considering the triples $(P_i,P_j,P_k)$ or $(P_j,P_i,P_k)$ we see that 
\[G_{t,n}-Z_n \leq \sum_{\substack{i,j,k \\ \mbox{\scriptsize{\textit{pairwise distinct}}}}} 1_{\{P_j \in C^{(n)}(P_i) , P_k \in C^{(n)}(P_i)\}} .\]
 Thus, taking expectation gives (\ref{inequality: pairs and triples}).
 
If $q_1,q_2 \in C^{(n)}(p)$ then $|q_1-q_2|=O(n^{-3/4})$ and using Hadamard-Fischer inequality and (\ref{appro: Rho2}) we obtain 
\[\mathring{\rho}_{3}^{(n)} (p,q_1,q_2) \leq \mathring{\rho}_{1}^{(n)} (p) \mathring{\rho}_{2}^{(n)} (q_1,q_2)=\frac{n}{4\pi} \mathring{\rho}_{2}^{(n)} (q_1,q_2)=O(n^{5/2}).\]
By integrating on the domain of area $4\pi |C^{(n)}(p)|^2=O(n^{-3})$, we conclude $\E |G_{t,n} - Z_n| \to 0$ as n goes to infinity. 
\end{proof}
Our goal is to prove $G_{t,n} \todist \mbox{Poisson}(\frac{x^4}{64})$ where $t=\frac{x}{n^{3/4}}$. Thanks to the Lemma \ref{lem: 3 close points}, it thus suffices to show that $Z_n \todist \mbox{Poisson}(\frac{x^4}{64})$. Denote $\displaystyle \tilde{\rho}_{k}^{(n)}(p_1,\ldots,p_k)$ as the $k$-point correlation function $\tilde{\mathcal{X}}^{(n)}$. From (\ref{eqn: def of rho}) we have 
\[\E \frac{Z_n!}{(Z_n-k)!}=\int_{(\Sphere)^{k}} \tilde{\rho}^{(n)}_k(p_1,\ldots , p_k) \,\dd \nu(p_1) \ldots \,\dd \nu(p_k).\]
So using the moment method it suffices to show that for every $k \geq 1$
\begin{equation} \label{Minimum spacing: Integral}
\int_{(\Sphere)^{k}} \tilde{\rho}^{(n)}_k(p_1,\ldots , p_k) \,\dd \nu(p_1) \ldots \,\dd \nu(p_k) \underset{n\to\infty}{\longrightarrow} \left(\frac{x^4}{64}\right)^k.
\end{equation}
( The $k$-th factorial moment of the Poisson distribution with mean $\lambda$ is equal to $\lambda^k$.)
\begin{proof}[Proof of Theorem \ref{thm: Minimum spacing}]
Let $p_1,\ldots , p_k$ be fixed distinct elements in $(\Sphere)^k$. First we show that
\begin{equation} \label{Minimum spacing: Pointwise}
\tilde{\rho}^{(n)}_k(p_1,\ldots , p_k) \underset{n\to\infty}{\longrightarrow} \left(\frac{x^4}{256 \pi}\right)^k.
\end{equation}
For $n$ large enough we make assume that for $i \neq j$ 
\[|p_i-p_j| >4t.\]
From inclusion-exclusion argument, we have (see \cite{Soshnikov2})
\begin{multline}\label{eqn:inclusionExclusion}
\tilde\rho^{(n)}_k(p_1,\dots,p_k)= \\
\frac{1}{2^k} \sum_{m=0}^{n-2k} \frac{(-1)^m}{m!}\int_{C^{(n)}(p_k)} \dots
\int_{C^{(n)}(p_1)}
\int_{\big(\bigsqcup_{i=1}^{k} C^{(n)}(p_i) \cup C^{(n)}(q_i)\big)^m} \\
\mathring{\rho}^{(n)}_{2k+m}(p_1,q_1,\dots,p_k,q_k,r_1,\dots,r_m)\dd \nu(r_1)\dots\dd \nu(r_m) \, \dd \nu(q_1) \ldots \dd \nu(q_k) .
\end{multline}
First consider the $m=0$ case. From determinantal formula, we have
\begin{equation} \label{eqn: Rho2k} 
\mathring{\rho}^{(n)}_{2k}(p_1,q_1,\ldots,p_k,q_k)=\det_{1 \leq i,j \leq k} 
\begin{pmatrix}
\mathring{K}^{(n)}(p_i,p_j) & \mathring{K}^{(n)}(p_i,q_j) \\[1mm]
\mathring{K}^{(n)}(q_i,p_j)&\mathring{K}^{(n)}(q_i,q_j) 
\end{pmatrix}.
\end{equation}
Consider the $(i,j)$-th $2 \times 2$ block of above determinant. If $i \neq j$ from (\ref{eqn: norm of kernel}) all terms of this block are exponentially small in $n$ where $q_j \in C^{(n)}(p_j), 1 \leq j \leq k$. Also for $i=j$, from (\ref{appro: Rho2}) the determinant of $(i,i)$-th $2 \times 2$ block is
\begin{equation*}
\mathring{\rho}^{(n)}_2(p_i,q_i)=\det
\begin{pmatrix}
\frac{n}{4\pi} & \mathring{K}^{(n)}(p_i,q_i)   \\[2mm]
\mathring{K}^{(n)}(q_i,p_i)&\frac{n}{4\pi}
\end{pmatrix}
=O(n^{3/2}).
\end{equation*}
So from the expansion of the determinant in (\ref{eqn: Rho2k}) over all permutations of length $2k$ we conclude that only the terms  contain the entries in the diagonal $2 \times 2$ blocks can have a non-zero limit. Note that $|\mathring{K}^{(n)}| \leq \frac{n}{4\pi}$ and the integration domain of $\mathring{\rho}^{(n)}_{2k}$ has size $O(n^{\frac{-3k}{2}})$. Thus from (\ref{eqn: point pair statistics}) and (\ref{expectation of point pair statistics}) we have
\begin{align*}
\int_{C^{(n)}(p_k)} \dots
\int_{C^{(n)}(p_1)} \mathring{\rho}^{(n)}_{2k}(p_1,q_1,\dots ,p_k,q_k)\, \dd \nu(q_1) \ldots \dd \nu(q_k) \\
=o(1)+\prod^{k}_{i=1} \int_{C^{(n)} (p_i)} \mathring{\rho}^{(n)}_{2}(p_i,q_i)\, \dd \nu(q_i) \to \left(\frac{x^4}{128 \pi}\right)^k
\end{align*}
as $n \to +\infty$.
By Hadamard-Fischer inequality we conclude that the contribution of the terms corresponding to $m \geq 1$ in (\ref{eqn:inclusionExclusion}) is bounded by
\begin{align*}
\Bigg(\int_{C^{(n)}(p_k)} \dots \int_{C^{(n)}(p_1)} \mathring{\rho}^{(n)}_{2k}(p_1,q_1,\dots,p_k,q_k) \, \dd \nu(q_1) \ldots \dd \nu(q_k) \Bigg) \times \nonumber \\ 
\sum_{m=1}^{n-2k} \frac{1}{m!}
\bigg(\int_{\bigsqcup_{i=1}^{k} \big( C^{(n)}(p_i) \cup C^{(n)}(q_i) \big)}
\mathring{\rho}^{(n)}_{1}(r_1) \dd \nu(r_1)\bigg)^m  
\end{align*}
The integration domain $\bigsqcup_{i=1}^{k} \big( C^{(n)}(p_i) \cup C^{(n)}(q_i) \big)$ has size $O(n^{-3/2})$ and $\mathring{\rho}^{(n)}_{1}(r_1)=\frac{n}{4\pi}$. Thus the second term of the above product goes to zero as $n \to +\infty$. Also the first factor of the above product is just the $m=0$ case, which converges. Thus the whole expression goes to zero as $n \to +\infty$ and (\ref{Minimum spacing: Pointwise}) is obtained. 

Also for all $(p_1, \ldots ,p_k) \in (\Sphere)^k$ we have
\begin{align*}
\tilde{\rho}^{(n)}_k(p_1,\ldots , p_k) &\leq 
\int_{C^{(n)}(p_k)} \dots \int_{C^{(n)}(p_1)} \mathring{\rho}^{(n)}_{2k}(p_1,q_1,\dots ,p_k,q_k)\, \dd \nu(q_1) \ldots \dd \nu(q_k) \\
&\leq \prod^{k}_{i=1} \int_{C^{(n)} (p_i)} \mathring{\rho}^{(n)}_{2}(p_i,q_i)\, \dd \nu(q_i) =\left(\frac{x^4}{128 \pi}\right)^k
\end{align*}
 Finally, From (\ref{Minimum spacing: Pointwise}) and the dominated convergence theorem, one obtains (\ref{Minimum spacing: Integral}), and the claim follows.
\end{proof}

\end{document}